\documentclass[12pt,reqno]{amsart}
\usepackage{amscd,graphicx}
\usepackage{amssymb}
\usepackage{color}
\usepackage{tikz-cd}
\usepackage[colorlinks,
    linkcolor={red!50!black},
    citecolor={blue!50!black},
    urlcolor={blue!80!black}]{hyperref}

\newcommand{\lto}{\longrightarrow}
\newcommand{\wh}{\widehat}
\newcommand{\wt}{\widetilde}
\newcommand{\sm}{\smallsetminus}

\newcommand{\al}{\alpha}
\newcommand{\be}{\beta}

\newcommand{\de}{\delta}
\newcommand{\ep}{\varepsilon}
\renewcommand{\th}{\theta}
\newcommand{\la}{\lambda}
\renewcommand{\phi}{\varphi}
\newcommand{\si}{\sigma}
\newcommand{\om}{\omega}
\newcommand{\Ga}{\Gamma}
\newcommand{\La}{\Lambda}

\newcommand{\De}{\Delta}
\newcommand{\Si}{\Sigma}
\newcommand{\ZZ}{{\mathbb Z}}
\newcommand{\CC}{{\mathbb C}}

\newcommand{\ft}{\mathfrak t}
\newcommand{\R}{{\rm R}}
\newcommand{\PR}{{\rm PR}}

\newcommand{\Hom}{\operatorname{Hom}\,}

\newcommand{\Aut}{\operatorname{Aut}}
\newcommand{\Diff}{\operatorname{Diff}}
\newcommand{\Fix}{\operatorname{Fix}}
\newcommand{\Ad}{\operatorname{Ad}}
\newcommand{\Stab}{\operatorname{Stab}}
\newcommand{\lk}{\operatorname{lk}\,}
\newcommand{\diag}{\operatorname{diag}}

\newcommand{\del}{\partial}

\newtheorem{theorem}{Theorem}[section]
\newtheorem{lemma}[theorem]{Lemma}
\newtheorem{proposition}[theorem]{Proposition}

\newtheorem{corollary}[theorem]{Corollary}
\newtheorem{conjecture}[theorem]{Conjecture}

\theoremstyle{remark}
\newtheorem{remark}[theorem]{Remark}
\newtheorem{example}[theorem]{Example}

\begin{document}

\title[The ${SU(N)}$ Casson-Lin invariants for links]
{The $\boldsymbol{SU(N)}$ Casson-Lin invariants for links}
\author{Hans U. Boden}
\address{Mathematics \& Statistics, McMaster University, Hamilton, Ontario} 
\email{boden@mcmaster.ca}
\thanks{The first author was supported by a grant from the Natural Sciences and Engineering Research Council of Canada.
The second author was supported by a CIRGET postdoctoral fellowship.}

\author{Eric Harper}
\address{Mathematics \& Statistics, McMaster University, Hamilton, Ontario}
\email{eharper@math.mcmaster.ca}


\date{\today}
\begin{abstract}
We introduce the $SU(N)$ Casson-Lin invariants  for links $L$ in $S^3$ with more than one component.
Writing $L = \ell_1 \cup \cdots \cup \ell_n$, we require as input an $n$-tuple
$(a_1,\ldots, a_n) \in \ZZ^n$ of \emph{labels}, where $a_j$ is associated with $\ell_j$.
 The $SU(N)$ Casson-Lin invariant, denoted $h_{N,a}(L)$, gives an algebraic count
of certain projective $SU(N)$ representations of the link group $\pi_1(S^3 \sm L)$,
and the family $h_{N,a}$ of link invariants gives  a natural extension of the $SU(2)$ Casson--Lin invariant, 
which was defined for knots by X.-S. Lin and for 2-component links by Harper and Saveliev. 
We compute the invariants for the Hopf link and more generally for chain links, and we show that, under
mild conditions on the labels $(a_1, \ldots, a_n)$, the invariants $h_{N,a}(L)$ vanish whenever $L$ is a  split link.
\end{abstract}
\maketitle


\section*{Introduction} 

The goal of this paper is to construct $SU(N)$ Casson-Lin invariants $h_{N,a}(L)$ for oriented links $L$ in $S^3$.  These 
invariants are defined as a signed count of conjugacy classes of certain  irreducible
 {projective} $SU(N)$ representations of $\pi_1 (S^3 \sm L)$ with a non-trivial $2$-cocycle.
Given an oriented link $L$ with $n$ components, the $2$-cocycle is determined by an $n$-tuple $a=(a_1, \ldots, a_n) \in \ZZ^n$ of labels, and 
the choice of labels is made so that the resulting $2$-cocycle is nontrivial. This is critical in what follows because it prohibits the
existence of reducibles, see Proposition \ref{NoReducibles}. 
We denote the resulting algebraic count as $h_{N, a}(L)$, and the following theorem is the main result of this paper.

\smallskip\noindent
{\bf Theorem.}
Suppose $L \subset S^3$ is an oriented $n$-component link  with $n\geq 2$ and $a=(a_1,\ldots, a_n)$ is an allowable $n$-tuple of labels.
Then the integer $h_{N, a}(L)$ is a well-defined invariant of $L$.

\medskip \noindent
We briefly outline how the above theorem is established. By Alexander's Theorem \cite{A}, every link $L \subset S^3$
can be realized as the closure $L=\wh{\si}$ for some braid $\si \in B_k.$ 
 The braid group $B_k$ acts naturally  on the free group $F_k$, and this
 induces an action on the space  of $SU(N)$ representations of $F_k,$ which we denote as
 $$R_k = \Hom(F_k, SU(N)) = SU(N) \times \cdots \times SU(N).$$

We extend this action to the wreath product $\ZZ_N \wr B_k$ as follows.
Identifying $\ZZ_N$ with the center of $SU(N)$, then for
$\ep = (\ep_1,\ldots, \ep_k) \in (\ZZ_N)^k$ and $X=(X_1,\ldots, X_k) \in R_k$,  
we set $(\ep,\si)(X) = (\ep_1 \si(X)_1,\ldots, \ep_k \si(X)_k).$  
This extends the braid group action on $R_k$   to an action of $\ZZ_N \wr B_k$, and 
in fact every fixed point $\Fix(\ep \si) \subseteq R_k$ can be identified with  a
 projective representation of the link group $G_L=\pi_1(S^3 \sm L).$
 
The key result is Proposition \ref{NoReducibles}, which shows that every element $X \in \Fix(\ep \si)$ is irreducible.
Consequently, writing $R^*_k \subset R_k$ for the subspace of irreducible $SU(N)$ representations, 
Proposition \ref{NoReducibles} implies that the  graph $\Ga^*_{\ep \si}$ and the diagonal $\De^*_k$ 
intersect in a compact subset of $R^*_k \times R^*_k$.
It follows that one can arrange transversality of the intersection $\Ga^*_{\ep \si} \cap \De^*_k$ 
by a compactly supported isotopy, and using natural orientations
on the quotients $\wh{\Ga}_{\ep\si} = \Ga^*_{\ep \si}/PU(N)$ and
 $ \wh{\De}_k= \De^*_k/PU(N)$, we define $h_{N,a}(\ep \si)$ as the 
 oriented intersection number of $\wh{\Ga}_{\ep\si}$ and $ \wh{\De}_k$. 
 Our main result is then established by showing that $h_{N,a}(\ep \si)$ is independent of 
 the choice of  compatible $k$-tuple $\ep=(\ep_1,\ldots, \ep_k)$  (Proposition \ref{indepofep}), 
 and that it is  invariant under the two Markov moves (Propositions \ref{markov1} and \ref{markov2}). 
It follows that $h_{N, a}(L)$ gives a well-defined invariant of the link $L \subset S^3$.
  
One of the virtues of this approach is that it leads to a direct method for computing the invariants, 
and we illustrate this by computing $h_{N,a}(L)$ for the Hopf link and for chain links (Propositions \ref{prop-hopf} and \ref{prop-chain}) and by showing that the invariants
vanish for split links (Proposition \ref{SplitLinks}).

\subsection*{Gauge Theory}
One motivation for defining link invariants in terms of the $SU(N)$ representations of the link group is that
these representations can be identified with flat connections on a principal $SU(N)$ bundle over the link exterior, which allows for a gauge theoretic interpretation.  
This approach was originally used by Taubes to interpret Casson's invariant $\la(\Si)$ of homology $3$-spheres $\Si$ in terms of flat $SU(2)$ connections \cite{Taubes}, and
using similar ideas, Floer defined $\ZZ_8$-graded groups $HF_*(\Si)$ called
the instanton Floer homology and
whose Euler characteristic equals the Casson invariant  \cite{Floer}.

The Casson-Lin invariants can also be interpreted gauge theoretically, as we now explain. 
The Casson-Lin invariant was originally defined in \cite{Lin}   by X.-S. Lin, who described an invariant $h(K)$ of 
knots $K \subset S^3$ as an algebraic count of conjugacy classes of \emph{tracefree} 
irreducible $SU(2)$ representations of $\pi_1(S^3 \sm K)$ and proved that $h(K) = {\rm sign}(K)/2$, half the signature of $K$.
In \cite{Herald}, C. Herald used  gauge theory to  define an extended Casson-Lin invariant
$h_\al(K)$ for knots $K \subset \Si^3$ in homology 3-spheres which allows for 
more general meridional trace conditions, and
he generalized Lin's formula by showing that $h_\al(K) ={\rm sign}_\al(K)/2$, half the  Tristram-Levine $\al$-twisted signature of $K$. 
(Similar results were obtained by M. Heusener and J. Kroll in \cite{HK}.)
O. Collin and B. Steer then used moduli spaces of orbifold connections to  
define an associated Floer homology theory for knots whose Euler characteristic equals $h_\al(K)$ in \cite{Collin-Steer},
and P. Kronheimer and T. Mrowka further developed the instanton Floer homology theory of  knots in \cite{KM1}, 
and they used it to prove a strong nontriviality result for Khovanov homology in \cite{KM2}. 
    
The second author and N. Saveliev used projective $SU(2)$ representations
to extend the Casson-Lin invariant to  2-component links $L$ in $S^3$ in \cite{HS1},
and they showed that $h(L) = \pm lk (\ell_1,\ell_2)$, the linking number of  $L=\ell_1 \cup \ell_2$. 
They gave a gauge theoretic description of the invariant $h(L)$ in \cite{HS2}, 
where they also described Floer homology groups with Euler characteristic equal to $h(L)$.

In view of all of these results, it is natural to ask whether the $SU(N)$ Casson-Lin invariants introduced here can also be interpreted gauge theoretically.  
We plan to address this question  in a future article using moduli spaces of projective $SU(N)$ representations (cf. \cite{RS}). 
We hope to use this  approach to extend the Casson-Lin invariants $h_{N,a}(L)$ to links $L \subset \Si^3$ in homology 3-spheres 
and to describe corresponding Floer homology groups.
In particular, we expect this approach will help clarify the relationship between the invariants $h_{N,a}(L)$ studied here 
and the $SU(N)$ instanton Floer groups constructed by Kronheimer and Mrowka in \cite{KM1}. 
It is possible that this approach will also shed light on other interesting questions, 
such as whether and how  the Casson-Lin invariants are related to classical link invariants, such as the higher linking numbers.
 
\medskip 
We give a brief outline of the contents of this paper.    
In Section \ref{sec1}, we introduce the notation for braids $\si \in B_k$, links $L \subset S^3$, and $SU(N)$ representations that is used throughout the article. In Section \ref{sec2}, we introduce allowable labels $(a_1,\ldots, a_n)$ for a given $n$-component link $L \subset S^3$, and compatible  $k$-tuple $(\ep_1,\ldots, \ep_k)$ for a braid $\si \in B_k$ with closure $L$. We also introduce projective $SU(N)$ representations of the link group $G_L$ and establish irreducibility of elements of $\Fix(\ep \si)$.
In Section \ref{sec3}, we define the invariant $h_{N,a}(L)$ as an oriented intersection number and prove it is independent of the various choices involved.
 In Section \ref{sec4}, we calculate the invariants $h_{N,a}(L)$ for the Hopf link and   the $n$-component chain link, and we prove a general vanishing result for the invariants   for split links. 

\section{Braids and representations} \label{sec1}
In this section,
we introduce the results for braids, links, and $SU(N)$ representations that will be used in throughout the article.

\subsection{The braid group}
We denote by $B_k$ the group of geometric braids on $k$-strands with
standard generators  $\si_1, \ldots ,\si_{k-1}$ and 
relations $\si_i \si_j = \si_j \si_i$ for $|j-i| > 1$ and
$\si_i \si_{i+1} \si_i = \si_{i+1} \si_i \si_{i+1}.$
The generators are depicted in Figure \ref{braid-gens} below, and note that in this paper we follow Convention 1.13 of \cite{KT}.

\bigskip
\begin{figure}[h]
\includegraphics[scale=0.8]{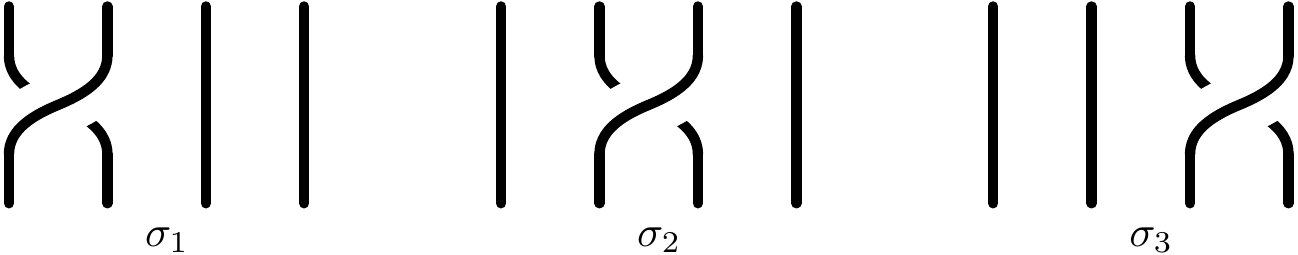} 
\caption{The three generators of  $B_4$}
\label{braid-gens}
\end{figure}

Each braid $\si \in B_k$ determines a permutation, and the resulting map $B_k \to S_k$, 
which sends the generator $\si_i$ to the transposition $\bar \si_i = (i, i+1)$, is a surjection.
Given $\si \in B_k,$ we let $\bar{\si} \in S_k$ denote the corresponding
permutation. Under this map, the symmetric group $S_k$ acts on the set $\{1,\ldots, k\}$ on the right.
For $ j \in \{1,\ldots, k\}$, we write $(i)^{\bar{\si}}$ for the image of $i$ under $\bar{\si} \in S_k.$

Let $F_k$ be the free group with free generating set  $x_1, \ldots, x_k$. 
There is a natural right action of the braid group $B_k$ on $F_k$ defined
by setting $\si_i\colon F_k \to F_k$ to be the map 
\[
\begin{array}{lll} 
 x_i & \mapsto & x_{i+1} \\
 x_{i +1} & \mapsto & (x_{i+1})^{-1}\, x_i\; x_{i+1} \\
 x_j & \mapsto & x_j, \;\; j \neq i ,i +1.
\end{array}
\]
This action defines 
a faithful representation
$\varrho\colon B_k \lto \Aut(F_k) $,
and we use it to identify $B_k$ with its image in $\Aut(F_k)$ under $\varrho.$ 
As this is a right action, we will use $x_i^{\si}$ to denote the image of $x_i$ under $\si \in B_k$.

\begin{example} \label{first-ex-page}
We explain how to read the action of a braid, which is explained in section 2.4 of \cite{FRR} for the left action, and we present the details for the right action. 

\begin{figure}[h]
\includegraphics[scale=0.85]{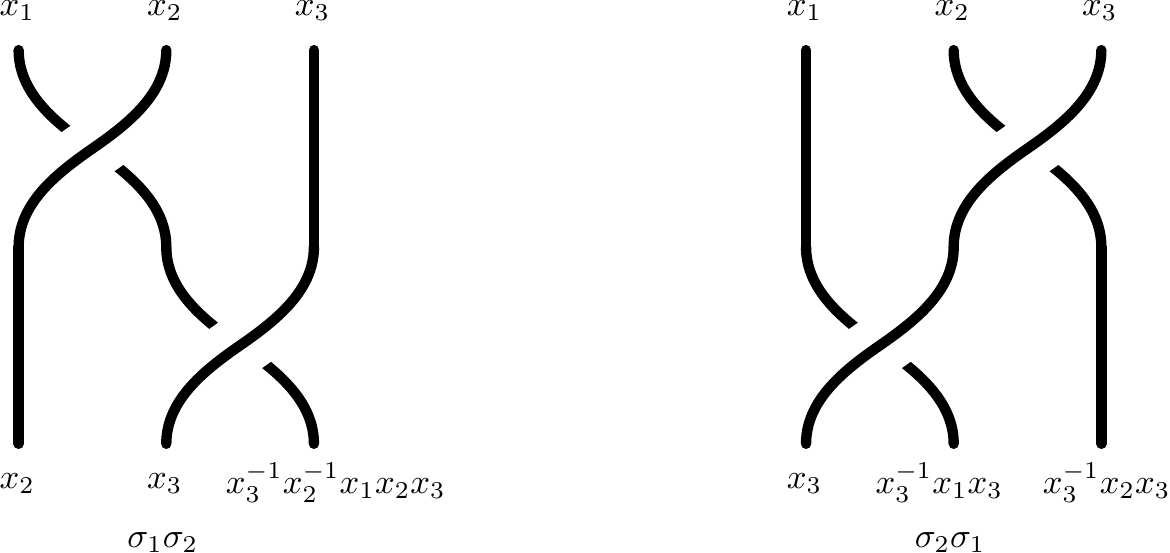} 
\caption{Reading the action of the braids $\si_1 \si_2$ and $\si_2 \si_1$}
\label{braid-read}
\end{figure}

The basic idea is to view the free group $F_k$ as the fundamental group of a 2-disk with $k$ punctures and keep track of basepoints as you move the disk vertically, letting the punctures move along the braid. 
Specifically,  label the top strands $x_1, \ldots, x_k$ from left to right. Then push the labels down, inserting a Wirtinger relation at each crossing. At the bottom of the braid the strands will be labeled by words $w_1, \ldots, w_k$ in $x_1, \ldots, x_k$, and the right action of $\si$ is given by the automorphism sending $x_i$ to $x_i^{\si}:=w_i$.

Using Figure \ref{braid-read}, we determine the actions of $\si_1 \si_2$ and $\si_2 \si_1$ on $F_3 = \langle x_1,x_2,x_3\rangle$ to be given by    
$$\begin{cases} x_1^{\si_1\si_2} = x_2, \\ x_2^{\si_1\si_2} = x_3, \\ x_3^{\si_1\si_2} = x_3^{-1}x_2^{-1}x_1 x_2 x_3,
\end{cases}  \quad \text{ and } \qquad\qquad
\begin{cases}  x_1^{\si_2\si_1} = x_3, \\ x_2^{\si_2\si_1} = x_3^{-1} x_1 x_3,\\ x_3^{\si_2\si_1} = x_3^{-1}x_2 x_3.\end{cases}$$
 \end{example}

We point out two facts about the action of $B_k$ on $F_k$, both of which are easily verified 
for each generator.
Firstly, for any $\si \in B_k$, 
$\si$ acts by conjugation and permutation on the generating set $x_1, \ldots, x_k$ for $F_k$. Indeed, 
\begin{equation}\label{BraidConj}
x_i^{\si} = w\, x_{(i)^{\bar{\si}}}\, w^{-1},
\end{equation}
where $w \in F_k$ is some word depending on $\si$ and $i$.
Secondly,  every braid $\si \in B_k$ preserves the product $x_1 \cdots x_k$, 
\begin{equation}\label{Product}
(x_1 \cdots x_k)^\si = x_1 \cdots x_k.
\end{equation}

\subsection{The group of a link}

Every link $L$ in $S^3$ can be realized as the closure $L=\wh\si$ of a braid $\si$.
We regard $L$ as an oriented link, where the strands of the braid $\si$ are oriented in the downward direction.
The link group $G_L=\pi_1 (S^3 \sm L)$ admits a standard presentation
\begin{equation}\label{G}
G_L = \pi_1 (S^3 \sm \wh\si) = \langle\,x_1,\ldots, x_k \mid x_i = x_i^\si,
\; i = 1,\ldots, k\,\rangle.
\end{equation}

The number of components of the link $L = \wh\si$ is the number 
of disjoint cycles in the permutation $\bar\si$. We will be interested in $n$-component 
links, that is, the closures of braids $\si$ with 
\begin{equation}\label{Cycles}
\bar\si = (i_1, \ldots, i_{k_1})(i_{k_1 + 1}, \ldots, i_{k_2}) \cdots (i_{k_{n-1}+1}, \ldots, i_{k_n}),
\end{equation}
where $1 \leq k_1 < k_2 < \cdots< k_n=k$.
We define multi-indices $I_1,I_2, \ldots ,I_n$ by setting 
 $I_j = \{ i_{k_{j-1}+1}, \ldots, i_{k_j} \}$ for $j=1,\ldots, n$, and  we
denote $\bar \si = (I_1)\cdots(I_n)$.  
If $L = \ell_1 \cup \cdots \cup \ell_n$ is the closure of a braid $\si$,
we will assume that the cycles in the permutation $\bar\si =(I_1) \cdots (I_n)$ are written
correspondingly, so that the component $\ell_j$ of $L$ corresponds to the braid closure of
the strands in $I_j$.

\subsection{The special unitary group}

Consider the Lie group $SU(N)$ of unitary $N \times N$ matrices with determinant one.  
Recall that $SU(N)$ has real dimension   
$N^2-1$ and has center  
isomorphic to $\ZZ_N = \{ \om^d \mid d \in \ZZ \}$, where
$\om  = e^{2 \pi i/N}$. Notice that we are viewing $\ZZ_N$ as the subgroup of $U(1)$ consisting of $N$-th
roots of unity, and for this reason we view it as a multiplicative group
and identify it with the center of $SU(N)$ via the map defined by sending $\om^d \mapsto \om^d I.$
 
Since every matrix in $SU(N)$ is diagonalizable, conjugacy classes in $SU(N)$ are completely
determined by their eigenvalues when considered with multiplicities. Given $A\in SU(N)$ we denote its conjugacy class by $C_A.$
There is a unique conjugacy class $C_A$
which is preserved under multiplication by $\om=e^{2 \pi i/N}$,
and this is the conjugacy class of the diagonal matrix $A$ whose eigenvalues
are the $N$ distinct $N$-th roots of $(-1)^{N-1}$. Setting $\xi = e^{2 \pi i/2N}$, then $A$ is
given by the diagonal matrix ${\rm diag}(1, \om, \ldots, \om^{N-1})$ when $N$ is odd
and by ${\rm diag}(\xi, \om \xi , \ldots, \om^{N-1}\xi)$ when $N$ is even.
In either case, since
the eigenvalues of $A$ are all distinct, we see that the stabilizer of $A$ is
the standard maximal torus $T^{N-1}$ in $SU(N)$ and that $C_A \cong SU(N)/T^{N-1}$ is the variety of full flags in $\CC^N$
and has 
real dimension $N^2-N$. 

\subsection{$\boldsymbol{SU(N)}$ representations}
For a discrete group $G$, let $\R (G) =\Hom(G, SU(N))$ denote the variety of $SU(N)$ representations of $G$.
For convenience, we set $\R_k = \R (F_k)= SU(N) \times \cdots \times SU(N)$ to be the variety of $SU(N)$ representations
of the free group $F_k$. 
The faithful representation $\varrho\colon B_k \to \Aut(F_k)$ induces a representation 
\begin{equation}\label{Diff1}
\wt\varrho \colon B_k \lto \Diff(\R_k)
\end{equation}
given by $\wt \varrho(\si)(\al) = \al \circ \si$.  
We will often abuse notation and simply denote $\wt\varrho(\si)$ by $\si$.

\begin{remark}
Theorem 2.1 of \cite{Long} implies that $\wt \varrho$ is defined on $\Aut(F_k)$ and is faithful.
\end{remark}

\begin{example} \label{first-ex-page}
Consider $\si_1$ and $\si_2 \in B_3$. For  $X=(X_1,X_2,X_3) \in \R_3$, we have
$$ \si_1(X)=(X_2, X_2^{-1} X_1 X_2,X_3) \qquad \text{and}\qquad 
\si_2(X) = (X_1,  X_3, X_3^{-1} X_2 X_3).$$

Using this, one can easily compute that
$$\si_1^2 (X)= (X_2^{-1} X_1 X_2, X_2^{-1} X_1^{-1} X_2 X_1 X_2, X_3)$$

and further that
\begin{eqnarray*} \si_1 \si_2 (X)&=& (X_2, X_3, X_2^{-1}X_3^{-1}X_1 X_3 X_2)  \qquad \text{and}\\
\si_2\si_1(X) &=& (X_3, X_3^{-1}X_1 X_3, X_3^{-1} X_2 X_3).
\end{eqnarray*}
\end{example}
      
Using the standard presentation \eqref{G} of the link group, 
we notice that $\R(G_L)$ can be identified with $\Fix(\si) \subset \R_k$, 
\[
\R(G_L)=\{ (X_1, \ldots, X_k) \in \R_k \mid X_i = \si (X)_i \}.
\]

A $k$-tuple $(X_1,\ldots X_k) \in \R_k$ is called \emph{reducible} if it can be simultaneously conjugated by
an element of $SU(N)$ such that each $X_i$ has the form
\begin{equation} \label{eq-red}
X_i=
\left(
\begin{array}{ccc}
  A_i & 0    \\
  0    & B_i \\  
\end{array}
\right),
\end{equation}
where $A_i$ is a block of size $N_1$ and $B_i$ is a block of size
$N_2$, and $N_1 + N_2 = N$. 
A $k$-tuple $(X_1,\ldots X_k) \in \R_k$ is \emph{irreducible} if it is not reducible.

\subsection{The wreath product $\boldsymbol{{\ZZ}_N \wr B_k}$  }

The wreath product $\ZZ_N \wr B_k$ is the semidirect product of $B_k$ with $(\ZZ_N)^k$, where $B_k$
acts on $(\ZZ_N)^k$  by permutation.
In other words, $\ZZ_N \wr B_k$ consists of pairs $(\ep, \si) \in (\ZZ_N)^k \times B_k$,
and the group structure is given by  
\[
(\ep,\si)\cdot(\ep',\si')\;=\;(\ep \bar{\si}(\ep'),\,\si\si').
\]

Here, $\bar{\si}$ acts on $\ep' = (\ep'_1, \ldots, \ep'_k)$ by permutation, i.e. $\bar{\si}(\ep') = (\ep'_{(1)^{\bar{\si}}}, \ldots, \ep'_{(k)^{\bar{\si}}})$.
In particular, it follows that $\si(\ep X) = \bar{\si}(\ep) \si(X).$

We extend the representation \eqref{Diff1} to the representation
\begin{equation}\label{Diff2}
\wt\varrho \colon \ZZ_N \wr B_k \lto \Diff(\R_k)
\end{equation} 
defined by sending the pair $(\ep, \si)$
to the diffeomorphism $\ep \si \colon \R_k \to \R_k$, where 
\[
\ep\si(X) = (\ep_1 \si(X)_1,
\ldots,\ep_k \si(X)_k).
\]
Thus $\ep$ twists the coordinates of $\si(X)$ by elements of the center $\ZZ_N$.

\begin{example} For $X=(X_1,X_2,X_3) \in \R_3$ and $\ep =(\ep_1,
\ep_2,\ep_3) \in (\ZZ_N)^3$, we have 

\begin{eqnarray*}
(\ep\si_1)(X_1,X_2,X_3) &=& (\ep_1\, X_2, \ep_2\,X_2^{-1} X_1 X_2,\ep_3\,X_3) \qquad \text{and} \\
\si_1(\ep X) = \bar{\si}_1(\ep) \si_1(X) 
&=& (\ep_2 \,X_2, \ep_1\, X_2^{-1} X_1 X_2, \ep_1 X_1 , \ep_3 X_3).
\end{eqnarray*}
\end{example}

\section{Projective representations of the link group} \label{sec2}

Our goal in this paper is to define 
invariants of $L$, and we will do so by performing a signed count of certain irreducible
projective $SU(N)$ representations.

\subsection{Projective representations} 
Suppose $\si \in B_k$ is a braid whose closure $\wh\si$ is a link $L$ in $S^3.$
For any $k$-tuple $\ep = (\ep_1,\ldots, \ep_k) \in (\ZZ_N)^k$,
an element $X=(X_1, \ldots, X_k) \in SU(N)^k$ in $\Fix(\ep\si)$
determines a $PU(N)$ representation of the link group $G_L$,
i.e. a homomorphism $\wt \al\colon G_L \to PU(N)$.
To see this, note that for any $X \in \Fix(\ep\si)$, since
$\ep_i \si(X)_i = X_i$ holds in $SU(N)$ and $\ep_i \in \ZZ_N$ is central, the
equation $\si(\wt X)_i = \wt{X}_i$ holds for the $k$-tuple $\wt{X} \in PU(N)^k$, which shows
that $\wt{X}$ determines a representation $\wt\al \colon G_L \to PU(N)$.

Given a discrete group $G$, 
we define a \emph{projective representation} of $G$
to be a \emph{function} (not a homomorphism!) 
$\al\colon G \to SU(N)$ such that $\al(gh) \al(h)^{-1} \al(g)^{-1} \in \ZZ_N$  
for all $g,h \in G.$
For any projective representation $\al\colon G \to SU(N)$, its  
composition with the surjection $\Ad\colon SU(N) \to PU(N)$ gives rise to a representation $\wt\al=\Ad \al \colon G \to PU(N)$,
and thus every projective representation $\al\colon G \to SU(N)$ is the lift  
of an honest representation $\wt\al\colon G \to PU(N)$. Alternatively, any representation
$\wt\al\colon G \to PU(N)$ can be lifted to a projective representation $\al\colon G \to SU(N)$,
though the lift is generally not unique.

Given a projective representation $\al\colon G \to SU(N)$, we can associate a map
$c\colon G \times G \to \ZZ_N$ defined by $c(g,h) =\al(gh) \al(h)^{-1} \al(g)^{-1}$.
Notice that the map $c$ satisfies the condition
that  $c(gh,k) c(g,h) = c(g,hk) c(h,k)$ for all $g,h,k \in G,$
and hence $c$ is a 2-cocyle of $G$.

For a fixed 2-cocycle $c\colon G \times G \to \ZZ_N$, 
let $\PR_c(G)$ denote the set of projective representations $\al\colon G \to SU(N)$ whose associated 2-cocycle is $c$.
If $G$ is finitely generated with generating set $\{g_1,\ldots, g_k\}$, then
any projective representation $\al \in \PR_c(G)$ is completely determined
by the 2-cocycle $c$ and the elements $\al(g_1), \ldots, \al(g_k) \in SU(N)$, and in this way one can realize $\PR_c(G)$ as a subset 
of $ SU(N)^k$. It is a compact real algebraic variety.

\subsection{Allowable labels and compatible $\boldsymbol{k}$-tuples}  
\label{subsect2-2}
Given a link $L$ in $S^3$ with $n$ components,
we can write $L=\ell_1 \cup \cdots \cup \ell_n.$
An $n$-tuple $a=(a_1,\ldots, a_n) \in \ZZ^n$ of integers is called
\emph{allowable}  if the following
three conditions are satisfied:
\begin{enumerate}
\item[(i)] $0 \leq a_i < N$ for $i = 1\ldots, n$,
\item[(ii)] $d=\gcd(a_1, \ldots, a_n)$ is relatively prime to $N$,
\item[(iii)] $a_1 + \cdots + a_n$ is a multiple of $N$. 
\end{enumerate}
An allowable $n$-tuple $(a_1, \ldots, a_n)$ is called an $n$-tuple of \emph{labels} for $L$,
and $a_j$ is the label corresponding to  the $j$-th component
$\ell_j$ of $L$.

Suppose now that $L$ is the closure of a braid $\si \in B_k$, and write
the permutation $\bar \si$ as a product $(I_1)\cdots(I_n)$ of disjoint cycles
in such a way that  $I_j$ corresponds
to the $j$-th component $\ell_j$ of $L$.

Recall that $\om = e^{2 \pi i/N}.$
A  $k$-tuple $\ep = (\ep_1,\ldots,\ep_k) \in (\ZZ_N)^k$ for $\si$ is said to be compatible
with the choice of labels $(a_1,\ldots, a_n)$ of $L$ if it satisfies the conditions
 \begin{equation}\label{Ep}
\prod_{i \in I_j} \ep_i  = \om^{a_j} 
\end{equation}
for $j=1,\ldots, n$.
This effectively labels each strand of the braid $\si$ so that, upon closure of the braid, the $j$-th component $\ell_j$ of $L$
is assigned the number $a_j$ for its label. 
Note that with this choice $\ep\si$ also preserves condition \eqref{Product} since, by Equation \eqref{Ep} and condition (iii), we have
\begin{equation}\label{EpProduct}
\begin{split}(\ep\si)(X)_1\cdots (\ep\si)(X)_k &= (\ep_1 \cdots \ep_k)  X_1\cdots X_k \\
 &= (\om^{a_1} \cdots \om^{a_n}) X_1\cdots X_k = X_1\cdots X_k.
\end{split}
\end{equation}

\subsection{An obstruction to lifting}

For $X \in \Fix(\ep \si)$ we will show that the associated representation   
$\wt\al \colon G_L \to PU(N)$ does not lift to an $SU(N)$ representation.
Essential for this conclusion is that the $k$-tuple $\ep$  is compatible
with $a$, the choice of labels for $L$.
In particular, we use
Equation \eqref{Ep} and assumption (ii) to give a nonzero obstruction to lifting $\wt \al$
to an $SU(N)$ representation.

\begin{proposition}
The representation
$\wt \al \colon G_L \to PU(N)$ does not lift to an $SU(N)$ representation.
\end{proposition}

\begin{proof}
Lift $\wt \al$ arbitrarily to a map $\al \colon G_L \to SU(N)$.  
Since $\al$ is a lift of $\wt \al$, we see that $\al(x_i) = \eta_i X_i$ for some $\eta_i \in \ZZ_N$.    
Let $\eta=(\eta_1, \ldots, \eta_k) \in (\ZZ_N)^k$ be the corresponding $k$-tuple.

We assume that  
$\al$ is a representation.  This implies that $\eta X \in \Fix(\si)$.
Since $X$ is also a fixed point of $\ep\si$ we have that 
\[
\eta_i X_i =  \si(\eta_i X_i) = \eta_{(i)^{\bar \si}} \si(X)_i =
(\ep_i)^{-1} \eta_{(i)^{\bar \si}} X_i.
\]
By condition (ii), some $a_j \neq 0$, and we assume (wlog) that $a_1 \neq 0$.
Consider the component $\ell_1$ associated with the multi-index $I_1=(i_1, \ldots, i_{k_1})$,
then Equation \eqref{Ep} implies that 
\[
\eta_{i_1}= (\ep_{i_1})^{-1}\eta_{i_2}=(\ep_{i_1})^{-1} (\ep_{i_2})^{-1} \eta_{i_3} =
\cdots=(\ep_{i_1})^{-1}\cdots (\ep_{i_k})^{-1} \eta_{i_1} =  \om^{-a_1} \eta_{i_1},
\]
which is a contradiction since $\om^{-a_1} \neq 1$.
\end{proof}

\subsection{Irreducibility for elements in $\boldsymbol{\Fix(\ep\si)}$}

We will now show that for any 
allowable $n$-tuple $a=(a_1,\ldots, a_n) \in \ZZ^n$ of labels and compatible
$k$-tuple $\ep=(\ep_1,\ldots, \ep_k) \in (\ZZ_N)^k$,
every $X \in \Fix(\ep \si)$ is irreducible.
The key to the proof is condition (ii) on the labels. 

\begin{proposition}\label{NoReducibles}
If $X \in \Fix(\ep\si)$, then $X$ is  irreducible.
\end{proposition}

\begin{proof}
Suppose to the contrary that $X \in \Fix(\ep \si)$ is reducible, which means that up to conjugation,
we can assume 
$$X_i=
\left(
\begin{array}{ccc}
  A_i & 0    \\
  0    & B_i \\  
\end{array}
\right),$$
where $A_i$ has size $N_1$ and $B_i$ has size $N_2.$

The first step is to consider the component $\ell_1$ of $L$.
It is obtained by closing the strands of $\si$ associated with the cycle $I_1=(i_1, \ldots, i_{k_1})$ of $\bar\si.$ 
By \eqref{BraidConj}, 
there are words $W_1, \ldots, W_{k_1}$ in $X_1, \ldots, X_k$ such that
\[
\begin{split}
X_{i_1}&=\ep_{i_1} W_1 X_{i_2}  W_1^{-1} \\
&=(\ep_{i_1}\ep_{i_2}) W_1 W_2 \, X_{i_3} \, W_2^{-1} W_1^{-1}= \cdots  \\
&= (\ep_{i_1} \cdots \ep_{i_{k_1}}) W_1 \cdots W_{k_1} \, X_{i_1} \, W_{k_1}^{-1} \cdots W_1^{-1} \\
&= \om^{a_1} W X_{i_1} W^{-1}, 
\end{split}
\]  
where the last step follows by setting
$W =W_1 \cdots W_{k_1}$ 
and applying Equation \eqref{Ep}. 

Since $W$ is a word in the $X_i$, and each $X_i$ is block diagonal, it follows
that $W$ is also block diagonal so we can write 
$$W=
\left(
\begin{array}{ccc}
  P & 0    \\
  0    & Q \\  
\end{array}
\right).$$
Applying this to the equation above, we see that the following relationship must hold for the blocks, so
 \begin{equation}\label{Blocks}
A_{i_1}=\om^{a_1} P A_{i_1} P^{-1},
\end{equation}
and taking the determinant of  
both sides of \eqref{Blocks},  we see that 
\[
\det(A_{i_1}) = \om^{a_1 N_1} \det(A_{i_1}).
\]
Because $\det(A_{i_1}) \neq 0$, this implies $\om^{a_1N_1}=1$. 

Now repeat the argument for the other components of the link $L$. 
For the component $\ell_j$,
which is the one obtained by closing the strands of $\si$ associated with the cycle $I_j$,
Equation \eqref{Ep} implies that $\om^{a_j N_1}=1$, and we see this holds for each $j=1,\ldots, n.$
However, since $\om=e^{2 \pi i/N}$ is a primitive $N$-th root of unity,
this can only happen if $N$ divides $a_j N_1$ for each $j=1,\ldots, n.$
This  contradicts condition (ii) on the labels,
 and we conclude that each $X \in \Fix(\ep \si)$ is in fact irreducible. 
 \end{proof}

\begin{remark} We would like to thank the referee for the following observation. Suppose $L = \ell_1 \cup \cdots \cup \ell_n$ is a link and let $\La_i \cong \langle \mu_i\rangle \times \langle \la_i\rangle  \cong \ZZ \times \ZZ$ denote the $i$-th peripheral subgroup of $G_L$, where $\mu_i, \la_i$ denote the meridian and longitude of $\ell_i$. Given a representation $\wt\al \colon G_L \to PU(N)$, let $\om(\wt\al) \in H^2(G_L,\ZZ_N)$ denote the obstruction cocycle,  
which is related to the commutator pairing of the restriction $\wt \al|_{\La_i}$ as follows. If $\al \colon G_L \to SU(N)$ is
a set-theoretic lift of $\wt\al$, then the \emph{commutator pairing} of $\La_i$ is the map $c_i \colon \La_i \times \La_i \to \ZZ_N$ given by $c_i(x,y) =[\al(x),\al(y)]$. Since $\La_i$ is free abelian of rank two, 
$$ \th \colon H^2 (\La_i,\ZZ_N) \stackrel{\cong}{\lto} \Hom(\La_i \wedge \La_i, \ZZ_N) \cong \ZZ_N,$$
and  one can show that $\th(\om(\wt\al|_{\partial_i}))=c_i.$

In the previous proof, the element $W=W_1\ldots W_{k_1}$ is the image of the longitude $\la_1$ of $\ell_1$, thus our computation that   $[X_{i_1},W] = \om ^{a_1}$  determined the commutator pairing $c_1=\th(\om(\wt\al_{\partial_1}))$ by showing that $c_1(\mu_1,\la_1) = \om^{a_1}$. The labels $a_1,\ldots, a_n$ thus determine the commutator pairings associated to the peripheral subgroups $\La_1,\ldots, \La_n$. 
\end{remark}

\section{The link invariants} \label{sec3}
Throughout this section, we assume that $\si$ is a braid with closure $\wh \si =L$
a link with $n$ components $L=\ell_1 \cup \cdots \cup \ell_n$, and that $a=(a_1,\ldots, a_n)$ is an $n$-tuple
of allowable labels, with $a_j$ the label for the component $\ell_j$.

In this section we define $h_{N, a}(\ep \si)$ for compatible $k$-tuples 
$\ep$, and we show that it gives rise to an invariant of $n$-component links in $S^3$.

We define $h_{N, a}(\ep \si)$ as an algebraic count of certain projective $SU(N)$ representations
in $\Fix(\ep \si)$, namely those that satisfy the monodromy condition $X_i \in C_A$.
In other words, we require each $X_i$ to be in the conjugacy class of matrices with characteristic
polynomial $p_A(t) = t^N + (-1)^{N}.$

We will first show  that $h_{N, a}(\ep \si)$ is independent of choice of $\ep$,
and then we prove that $h_{N, a}(\ep \si)$ gives rise to a well-defined invariant of the underlying link $L$ 
 by showing that it is invariant under the Markov moves.
 
\subsection{The definition of $\boldsymbol{h_{N,a}(\ep \si)}$}
Recall that $A$ is the diagonal matrix consisting of
the $N$-th roots of $(-1)^{N-1},$
i.e.

$$A=\begin{cases} {\rm  diag}(1,\om, \ldots, \om^{N-1}) &\text{if $N$ is odd} \\
 {\rm  diag}(\xi,\om \xi, \ldots, \om^{N-1}\xi) &\text{if $N$ is even.}
 \end{cases}$$

We impose the following monodromy condition and restrict to $k$-tuples
lying in the subset $Q_k \subset \R_k$ given by
\[
Q_k = \{(X_1,\ldots,X_k) \in \R_k \mid  X_i \in C_A\,\}.
\]
Since $Q_k=(C_A)^k$ is a just a $k$-fold product of $C_A,$ 
we see that $Q_k$ is a manifold of dimension $k(N^2-N)$.

Let  $\De_k = \{ (X, X) \} \subset Q_k \times Q_k$ be the diagonal and 
$\Ga_{\ep \si} = \{ (X, \ep\si(X)) \} \subset Q_k \times Q_k$ be the graph
of $\ep \si$.  Notice that we can identify points in
the intersection $\De_k \cap \Ga_{\ep \si}$ with elements in 
$Q_k \cap \Fix(\ep \si)$.

For certain choices of labels, it will follow that $\Fix(\ep \si) \subset Q_k$,
i.e. that these monodromy conditions are automatically satisfied.
This will occur whenever the labels have the property  
that each $a_i$ is relatively prime to $N$. 
For a simple example, suppose $N$ is prime, $n$ is a positive multiple of $N$,
and $d$ is any positive integer less than $N$. 
Then one can easily verify that $a =(d,d,\ldots, d)$ is an allowable $n$-tuple of labels,
and the next result implies that $\Fix(\ep \si) \subset Q_k$ for any $k$-tuple $\ep=(\ep_1,\ldots, \ep_k)$ compatible with these labels. 

\begin{proposition}\label{A}  
Suppose $(a_1, \ldots, a_n)$ is an allowable $n$-tuple of labels such that 
each $a_j$ is relatively prime to $N$, and suppose $\ep=(\ep_1,\ldots, \ep_k) \in (\ZZ_N)^k$ is compatible with $a$.
Then $\Fix(\ep\si) \subset Q_k,$ i.e. if
$X \in \Fix(\ep\si)$, then each $X_j$ is conjugate to $A$. 
\end{proposition}

\begin{proof}
The condition on $a_j$ ensures that $\om^{a_j}$ generates $\ZZ_N$ for each $j=1,\ldots, n,$
where $\om=e^{2 \pi i/N}$.
Write the induced permutation $\bar\si = (I_1)\cdots (I_n)$ as a product of disjoint cycles,
where $I_j$ corresponds to the $j$-th component of $L =\wh\si$.
Then for any $i \in I_j,$ we can 
apply the same argument as used to prove Proposition \ref{NoReducibles}
to see that 
$$X_i=\om^{a_j} W X_i W^{-1}.$$  
Thus, the set of eigenvalues of $X_i$
is invariant under multiplication by $\om^{a_j},$ 
and this shows the eigenvalues of $X$ 
are given by the set
$$\{ \xi, \om^{a_j} \xi, \ldots, \om^{(N-1)a_j}\xi\} = \{ \xi, \om \xi, \ldots, \om^{N-1} \xi\}$$
for some $\xi$ satisfying $\xi^N = (-1)^{N-1}$.
When $N$ is odd, one can take $\xi = 1, $ and when $N$ is even,
one can take $\xi = e^{2 \pi i/2N},$ and this shows that $X_i$ is conjugate to $A$.

Alternative argument: consider the characteristic polynomial of both sides of the above equation we see that
$p_{X_i}(t) = p_{\om^{a_j} X_i}(t) = p_{X_i}(\om^{-a_j}t)$. Since $\om^{-a_j}$ has order $N$,
 $p_{X_i}(t)$ must be a polynomial in $t^N$, and indeed the only possibility
 is that $ p_{X_i}(t)= t^N + (-1)^N.$ 
\end{proof}

We define \[
H_k = \{ (X,Y) \in Q_k \times Q_k \mid X_1 \cdots X_k = Y_1 \cdots Y_k \},
\]
and we note that $H_k$ is  not a manifold because of the presence of reducibles.
Recall that $(X,Y) \in Q_k \times Q_k$ is called {\it reducible} if all $X_i$ and $Y_i$ can be 
simultaneously conjugated into block diagonal form as in \eqref{eq-red}.
We note that the subset $S_k \subset Q_k \times Q_k$ of reducibles is closed,
and that $(Q_k \times Q_k)^* = (Q_k \times Q_k) \sm S_k$ is an open manifold of dimension $2k(N^2-N).$

\begin{proposition}
The subset $H^*_k = H_k \sm S_k$ of irreducible representations 
is an open manifold of dimension $2k(N^2-N) - (N^2 -1)$.
\end{proposition}

\begin{proof}
Clearly $H_k^* = f^{-1}(I)$, where $f\colon(Q_k \times Q_k)^* \to SU(N)$ is the map defined by $f(X,Y) = X_1 \cdots X_k Y_k^{-1} \cdots Y_1^{-1}.$
We will show that $I$ is a regular value of $f$, i.e.\ that $df_{(X,Y)}$ is surjective for all $(X,Y) \in f^{-1}(I)$. 
It is enough to prove this statement for  the map
$f\colon Q_\ell^* \to SU(N)$ given by $f(X_1,\ldots, X_\ell) = X_1\cdots X_\ell.$
 
Clearly the matrix $A$, since it is diagonal,
lies on the standard maximal torus $T^{N-1} \subset SU(N)$ with Lie algebra  
$$\ft= \left\{ \begin{pmatrix} ia_1 && 0\\ & \ddots \\ 0&& ia_N \end{pmatrix} \mid a_1 + \cdots +a_N =0\right\}.$$
Since $A$ has the standard maximal torus as its stablizer group, we can identify the tangent space $T_A(C_A)$ with the orthogonal
complement $\ft^\perp$ in $su(N)$, which is the subspace
 $$\ft^\perp= \left\{ \begin{pmatrix} 0 & z_{12} & \ldots &z_{1N} \\ -\bar{z}_{12} &0 & \ddots & \vdots\\ \vdots& \ddots& \ddots & z_{N-1, N} \\
 -\bar{z}_{1N} & \ldots &-\bar{z}_{N-1, N} & 0 \end{pmatrix} \mid  z_{ij} \in \CC \right\}.$$
 
 There is a similar decomposition of $su(N)$ at each $X_i$ in the $\ell$-tuple $(X_1, \ldots, X_\ell)$.
 Because each $X_i$  has $N$ distinct eigenvalues, it 
lies on a unique maximal torus $T_i \cong T^{N-1}$ in $SU(N)$.
We let $\ft_{i} \subset su(N)$ denote the corresponding Lie subalgebra, which is
the Lie algebra of the stabilizer group of $X_i$.
Using the decomposition $su(N) = \ft_i \oplus \ft_i^\perp,$
we can identify the tangent space 
$T_{X_i} (C_A)$ with $ 
\ft_i^\perp X_i$, the  right translation  of the subspace $\ft_i^\perp \subset su(N)$ by $X_i$. 
It is helpful to note that, in terms of the specific
subspaces identified above, we have $\ft_i = \Ad_{P_i} \ft$ and $\ft_i^\perp = \Ad_{P_i} \ft^\perp$ for any matrix $P_i\in SU(N)$ such that $X_i = P_i A P_i^{-1}.$

Using the fact that $\ft_i$ is the Lie algebra of the stabilizer
subgroup of $X_i$, one can see that irreducibility of the $\ell$-tuple $(X_1, \ldots, X_\ell)$ is equivalent to 
the condition that $\ft_1 \cap \cdots \cap \ft_\ell = \{0\}$.
 
For $u_i \in \ft^\perp_i$, we set $x_i=u_i X_i \in \ft_i^\perp X_i = T_{X_i} (C_A).$  
Differentiating and using the fact that $X_1 \cdots X_\ell =I$, we obtain 
\begin{align*}
 \frac{d}{dt} & (X_1+t x_1)(X_2+tx_2) \cdots (X_\ell + t x_\ell) \vert_{t=0} \\
&=x_1 X_2 \cdots X_\ell + X_1 x_2 X_2 \cdots X_\ell + \cdots +X_1 \cdots X_{\ell-1} x_\ell \\
&= u_1 + X_2 u_2 X_2^{-1} + \cdots + (X_2 \cdots X_\ell) u_\ell (X_\ell^{-1} \cdots X_2^{-1}).
\end{align*}
In order to show that the map $df_X$ is onto, we claim that, given any $v \in su(N),$  we can 
find 
$u_i \in \ft_i^\perp$ for $i=1,\ldots, \ell$
such that
\begin{equation} \label{eq-tan}
v =  u_1 + X_2 u_2 X_2^{-1}  + \cdots +(X_2 \cdots X_\ell) u_\ell (X_\ell^{-1} \cdots X_2^{-1}).
\end{equation}

 Notice that we can solve \eqref{eq-tan} for any 
$$v \in \ \ft_1^\perp  \cap  \left(X_1\,  \ft_2^\perp  \, X_1^{-1}\right) \cap \cdots \cap \left(X_1 \cdots X_{\ell-1}\right) \ft_\ell^\perp \left(X_{\ell-1}^{-1} \cdots X_1^{-1}\right).$$
Notice further that since $\ft_i$ is the Lie algebra of the maximal torus containing 
$X_i$, we have $\ft_i \cap \ft_{i+1}= t_i \cap \left(X_i \, \ft_{i+1}\, X_i^{-1}\right)$. More generally, for any subspace $V \subset su(N),$
we have $\ft_i \cap V= t_i \cap \left( X_i \, V\, X_i^{-1}\right)$.
Repeated application gives that
\begin{eqnarray*}
\ft_1 \cap \cdots \cap \ft_\ell&=& \ft_1 \cap \cdots \cap \left(X_{\ell-1} \, \ft_\ell
\, X_{\ell-1}^{-1}\right) \\
&=& \ft_1 \cap \cdots \cap \left(X_{\ell-2} \, \ft_{\ell-1}
\, X_{\ell-2}^{-1}\right) \cap  \left(X_{\ell-2} X_{\ell-1}\,  \ft_\ell
 \, X_{\ell-1}^{-1} X_{\ell-2}^{-1}\right) \\
&\vdots& \\
&=&  \ft_1 \cap \left( X_1 \ft_2 X_1^{-1}\right) \cap \cdots \cap \left(X_1 \cdots X_{\ell-1} \right) \ft_\ell \left( X_{\ell-1}^{-1}\cdots X_1^{-1}\right).
\end{eqnarray*}
The condition of irreducibility implies that
$\ft_1 \cap \cdots \cap \ft_\ell=\{0\}$, and it then follows from the above that
\eqref{eq-tan} can be solved for any $v\in su(N)$. This concludes
the argument that $df_X$ is a surjection whenever the $\ell$-tuple $X = (X_1, \ldots, X_\ell)$
is irreducible.
\end{proof}

Since both $\De_k$ and $\Ga_{\ep \si}$ preserve the product $X_1 \cdots X_k$, see \eqref{EpProduct}, we can
restrict from $Q_k \times Q_k$ to $H_k$.
 
Now we are in a position to define the invariant $h_{N, a}(\ep \si)$. 
Set
$\Ga^*_{\ep\si} = \Ga_{\ep\si}\,\cap\,H_k^*$ and $\De_k^* 
= \De_k\,\cap\, H_k^*$. Since $S_k$ is closed, it follows that both $\Ga^*$ and $\De_k^*$ are open submanifolds of $H_k^*$ of
dimension $k(N^2-N)$.  

Both $\De_k$ and $\Ga_{\ep\si}$ are compact, and so is their intersection
$\De_k \cap \Ga_{\ep\si}$. Consequently, as Proposition \ref{NoReducibles}
implies that every point in this intersection is irreducible, we have the following result.

\begin{corollary}
The intersection $\De_k^*\,\cap\,\Ga_{\ep\si}^* \subset H_k^*$ is 
compact.
\end{corollary}

The group $PU(N)$ acts freely by conjugation on each of $H_k^*$, 
$\De_k^*$, and $\Ga_{\ep\si}^*$, and the quotients by this 
action are the
manifolds we denote as
\[
\wh H_k = H_k^*/PU(N),\quad \wh \De_k = \De^*_k/PU(N), \quad\text{and}
\quad \wh \Ga_{\ep\si} = \Ga^*_{\ep\si}/PU(N).
\]
Here, the dimension of $\wh H_k$ equals $2k(N^2-N) - 2(N^2 -1)$, and 
both $\wh \De_k$ and 
$\wh \Ga_{\ep\si}$ are half-dimensional submanifolds of $\wh H_k$.
Since the intersection $\wh \De_k\,\cap\,\wh \Ga_{\ep\si}$ is 
compact, we can isotope $\wh \Ga_{\ep\si}$ into a submanifold 
$\wt \Ga_{\ep\si}$ using an isotopy with compact support 
so that the intersection $\wh \De_k\,\cap\,\wt \Ga_{\ep\si}$ is transverse and
consists of finitely many points. Define 
\[
h_{N, a}(\ep\si) = \#_{\wh H_k}\,(\wh\De_k\,\cap\,\wt \Ga_{\ep\si})  
\]
as the oriented intersection number. We will describe the orientations in the
following subsection. The intersection number $h_{N, a}(\ep\si)$ is independent of the choice of 
isotopy of $\wh \Ga_{\ep\si}$, and we denote
\[
h_{N, a}(\ep\si) = \langle\,\wh \De_k, \wh \Ga_{\ep\si}\,\rangle_{\wh H_k}.
\]

\subsection{Orientations}
The following argument is similar to the one found in section 3.4 of \cite{HS1}, and we include
it here for completeness.

First, observe that the conjugacy class $C_A \subset SU(N)$ is orientable, which follows for instance 
by identifying it with a flag variety. So choose an orientation for $C_A$ and give $Q_k =  (C_A)^k$
and $Q_k \times Q_k$ the induced product orientations. 
The diagonal $\De_k$ and the graph $\Ga_{\ep \si}$
are naturally diffeomorphic to $Q_k$ via projection and so an orientation for $Q_k$ determines
orientations for both $\De_k$ and $\Ga_{\ep \si}$.

Using the standard orientation of $SU(N)$, we obtain an orientation on 
$H_k^* = f^{-1}(I)$ using the base-fiber rule. 
Since the adjoint action of $PU(N)$ on $C_A$ is orientation preserving, the quotients
$\wh{H}_k, \wh{\De}_k,$ and $\wh{\Ga}_{\ep \si}$ are all orientable, and we
orient them using the base-fiber rule.

Reversing the orientation of $C_A$ reverses the orientation of $Q_k$ only when $k$ is odd,
and in this case it reverses the orientations of both $\wh{\De}_k,$ and $\wh{\Ga}_{\ep \si}$
but it does not affect the oriented intersection number  $\langle\,\wh \De_k, \wh \Ga_{\ep\si}\,\rangle_{\wh H_k}$.
This shows that the intersection number is actually independent of the choice of orientation on
the conjugacy class $C_A.$

\subsection{Independence of ${\boldsymbol \ep}$}
The next result shows that $h_{N, a}(\ep \si)$ is independent of the choice
of $\ep$ compatible with $a$.

\begin{proposition}   \label{indepofep}
Fix a link $L$ with $n>1$ components and an allowable $n$-tuple $a=(a_1,\ldots, a_n)$ of labels.
Fix also a braid $\si \in B_k$ with closure $\wh\si = L$. If $\ep, \ep' \in (\ZZ_N)^k$ are 
$k$-tuples compatible with $a$, i.e. satisfying Equation \eqref{Ep},   
then $h_{N, a}(\ep \si) = h_{N, a}(\ep' \si)$.
\end{proposition}

\begin{proof}

We will define an orientation preserving automorphism  $\phi \colon \wh H_k \to \wh H_k$ such 
that $\phi(\wh \De_k) = \wh \De_k$ and $\phi(\wh \Ga_{\ep \si}) = \wh \Ga_{\ep' \si}$.
Write  the permutation 
\begin{equation} \label{cycles-2}
\bar{\si} = (i_1, \ldots, i_{k_1})(i_{k_1 + 1}, \ldots, i_{k_2}) \cdots (i_{k_{n-1}+1}, \ldots, i_{k_n})
\end{equation}
as a product of disjoint cycles as in Equation \eqref{Cycles} and
define $\de = (\de_1, \ldots, \de_k) \in (\ZZ_N)^k$ recursively with initial values
\begin{equation} \label{initial}
\de_{i_1} = 1 = \de_{i_{k_1+1}} =  \cdots = \de_{i_{k_{n-1}+1}}
\end{equation}
and by setting
\begin{equation} \label{induct}
\de_{(j)^{\bar \si}} \;=\; \de_j\,\ep_j\,(\ep'_j)^{-1}.
\end{equation}
Writing $\bar \si$ as a product of disjoint cycles as in \eqref{cycles-2} and noting that $\ep$ and $\ep'$ both satisfy Equation \eqref{Ep}, repeated application of
the recursion \eqref{induct}
shows that the definition of $\de = (\de_1, \ldots, \de_k)$ is compatible with the initial values taken in \eqref{initial}.

Define the diffeomorphism $\tau \colon Q_k \to Q_k$ by $$\tau (X) = \de X = (\de_1 X_1, \ldots, \de_k X_k).$$  Note that $\tau$ may be
orientation preserving or reversing.  Furthermore, $\tau$ preserves irreducibility and commutes with 
conjugation.  

Consider the product map $\tau \times \tau \colon Q_k \times Q_k \to Q_k \times Q_k$.  Observe
that $\tau \times \tau$ preserves the orientation of $Q_k \times Q_k$ and hence the induced map
$\phi \colon \wh H_k \to \wh H_k$ is orientation preserving.

Since $\tau$ may be orientation reversing, $\phi$ restricted to $\wh \De_k$ or $\wh \Ga_{\ep \si}$ 
may be orientation reversing.  The key observation is that if $\phi$ is orientation reversing on one, then it must be orientation reversing on the other.  
Hence, $\phi$ preserves the intersection number $h_{N, a} (\ep \si)$, 
\[
\langle \,\wh \De_k, \wh \Ga_{\ep \si} \, \rangle_{\wh H_k} = 
\langle \,\phi(\wh \De_k),\, \phi(\wh \Ga_{\ep \si}) \, \rangle_{\wh H_k}.
\]

Clearly, $\phi(\wh \De_k) = \wh \De_k$, so to finish off the proof we check that 
$\phi(\wh \Ga_{\ep \si}) = \wh \Ga_{\ep' \si}$, or 
 that the pair $(\de X,\de\ep\si(X)) \in \wh \Ga_{\ep' \si}$.
By the following calculation
\[
(\de X, \de \ep \si (X)) =
(\de X, \de \ep \si (\de^{N-1} \de X)) = 
(\de X, \de \ep \bar{\si}(\de^{N-1})\,\si(\de X)), 
\]
this will follow once we verify that $\de \ep \bar{\si}(\de^{N-1})= \ep'$. Since $\de_j^{N-1}= \de_j^{-1},$ this is equivalent to showing that $\de_j \ep_j =  \bar{\si}(\de_j) \ep'_j$ for $j=1,\ldots, k$, which follows directly from Equation \eqref{induct}, and this completes the proof of the proposition.
\end{proof}

\subsection{Independence under Markov moves}
Based on the previous result, we denote $h_{N, a}(\ep \si)$ by $h_{N, a}(\si)$ assuming that a choice of compatible $\ep$  
has been made. 
In this subsection, we show that  $h_{N, a}$ defines an invariant of $n$-component links, and
this is achieved by showing that $h_{N, a}(\si)$ is invariant under the Markov moves.

Recall that two braids $\si, \tau \in B_k$ have isotopic closures $\wh\si = \wh\tau$ if and only if $\si$ can be 
obtained from $\tau$ by a finite sequence of Markov moves; see 
for example \cite{Birman}.  The first Markov move replaces $\si \in 
B_k$ by $\xi^{-1} \si \xi \in B_k$ for $\xi \in B_k$, and the second
Markov move exchanges $\si \in B_k$ with $\si  \si_k^{\pm 1} 
\in B_{k+1}$. 

The following propositions give the $SU(N)$ analogues of the $SU(2)$ results in \cite[Propositions 4.2 \& 4.3]{HS1}
(cf. the proof of \cite[Theorem 1.8]{Lin}).

\begin{proposition}   \label{markov1}
The quantity $h_{N, a}(\si)$ is invariant under type 1 Markov moves.
\end{proposition}

\begin{proof}
Suppose $\si \in B_k$ is a braid with  
$$\bar\si= (I_1) \cdots (I_n) = (i_1,\ldots, i_{k_1}) \cdots (i_{k_{n-1}+1},\ldots, i_{k_n})$$
in multi-index notation. Given a braid $\xi \in B_k,$ let $\si'=  \xi^{-1} \si \xi$ 
and note that $\overline{\si'}$ has
the same cycle structure as $\bar{\si}$, in fact it is given by
$$\overline{\si'}= \left(I^\xi_1\right) \cdots \left(I^\xi_n\right) = \left((i_1)^{\bar{\xi}},\ldots, (i_{k_1})^{\bar{\xi}}\right) \cdots \left((i_{k_{n-1}+1})^{\bar{\xi}},\ldots, (i_{k_n})^{\bar{\xi}}\right).$$

We choose $\ep' \in (\ZZ_N)^k$ compatible with the given labels,
which means that $\ep'$ satisfies Equation \eqref{Ep} with respect to the braid $\si'$, namely
$$\prod_{i \in I^\xi_j} \ep'_i=\om^{a_j}$$
holds for $j=1,\ldots, k.$ 
Notice that if we define the $k$-tuple $\ep$ by setting $\ep_i = \ep'_{(i)^{\bar{\xi}}},$ then one can show that
$\ep$  satisfies Equation \eqref{Ep} with respect to 
$\bar\si= (I_1) \cdots (I_n)$, hence $\ep$ is also compatible with the given labels.

The braid $\xi$ determines a map $\xi \colon Q_k \to Q_k,$ and since it acts by permutation 
and conjugation on each of the factors in $Q_k = C_A \times \cdots \times C_A,$ the fact that $C_A$ is even-dimensional
implies that this map is orientation preserving. This induces the map $ \xi \times  \xi$ on $Q_k \times Q_k$
preserving irreducibility,  commuting with the adjoint action of $PU(n)$, and preserving Equation
\eqref{Product}, thus we obtain a well-defined orientation preserving map
$ \xi \times \xi \colon \wh{H}_k \to \wh{H}_k.$
 
Clearly, $( \xi \times \xi)(\wh{\De}_k) =  \wh{\De}_k$, so the diagonal is preserved,
and we consider  the effect of $ \xi \times  \xi$ on the graph $ \wh\Ga_{\ep' \si'}$.
If $(X, \ep' \si'(X)) \in \wh\Ga_{\ep' \si'},$
then
$$( \xi \times \xi)(X, \ep' \si'(X)) = ( \xi \times  \xi)(X, \ep' \xi^{-1} \si \xi (X)) = ( \xi(X), \xi(\ep') \si  \xi (X))
\in \wh\Ga_{\ep \si},$$
since $\xi(\ep')_i = \ep'_{(i)^{\bar{\xi}}} = \ep_i$.
Thus $(\xi \times \xi)(\wh\Ga_{\ep' \si'}) =  \wh\Ga_{\ep \si},$ and we see that
\begin{eqnarray*}
h_{N, a}(\si') =  \langle \,\wh \De_k, \wh \Ga_{\ep' \si'} \, \rangle_{\wh H_k} &=& 
\langle \,(\xi \times \xi) (\wh\De_k), (\xi \times \xi)( \wh \Ga_{\ep' \si'}) \, \rangle_{\xi \times \xi(\wh H_k)} \\
&= &  \langle \,\wh \De_k, \wh \Ga_{\ep \si} \, \rangle_{\wh H_k} = h_{N, a}(\si).
\end{eqnarray*}
\end{proof}

The next result is established using the same argument that is used to prove \cite[Proposition 4.3]{HS1}
and \cite[Theorem 1.8]{Lin}. We leave the details of the proof to the reader.
\begin{proposition}   \label{markov2}
The quantity $h_{N, a}(\si)$ is invariant under type 2 Markov moves.
\end{proposition}

\section{Computations} \label{sec4}
In this section, we perform computations of $h_{N,a}(L)$ for various links $L$ and we prove a vanishing condition for  
$h_{N,a}(L)$ for split links.

\subsection{The Hopf link and chain links}
The  chain link $L$  is obtained as the closure of the 
braid   $\si= \si_1^2 \si_2^2 \cdots \si_{n-1}^2 \in B_n.$ 
In this subsection, we  compute $h_{N,a}(L)$ for $L$ the Hopf link and the chain link with $N=n$ components. In particular, if $d$ is chosen relatively prime to $N$ and $a = (d, \ldots, d),$ then we will show that  $h_{N,a}(L)=0$ for the chain link with $n>2$ components. 
For $n=2,$ $L$ is just the Hopf link, which we denote by $H \subset S^3$. 
In  \cite{HS1}, it is proved that $h_{2,a}(H) = \pm 1$ for $a=(1,1)$. We generalize this by showing  that  $h_{N,a}(H) = \pm 1$ if $a=(d, N-d)$, where $d$ satisfies $1\leq d <N$ and is relatively prime to $N$.

The next result will be used repeatedly in the computations that follow.
 
\begin{theorem} \label{thm-pair}
Suppose $N \ge 2$ and set $\om = e^{2 \pi i/N}$ and $\xi =e^{2 \pi i/2N}$. 
Any pair of matrices $(X,Y) \in SU(N) \times SU(N)$ satisfying  
$[X,Y] = \om I$ is, up to  conjugation, given by 
$$X  = \begin{cases} \diag(1, \om, \ldots, \om^{N-1}) & \text{if $N$ is odd,} \\
\diag(\xi, \xi \om, \ldots, \xi \om ^{N-1}) & \text{if $N$ is even,}
\end{cases}$$
and $$Y = \begin{pmatrix}
0& \dots& &\pm1 \\
1& \ddots& &\vdots\\
 &\ddots &  \\
     0&  &1 &0 \end{pmatrix}.$$ 
The pair $(X,Y) \in SU(N) \times SU(N)$ determines an irreducible projective $SU(N)$ representation of the free abelian group $\ZZ \oplus \ZZ$ of rank 2.   
\end{theorem}

\begin{proof}
First notice that $XYX^{-1}Y^{-1} = \om I$ if and only if $Y^{-1} X Y = \om X.$ 
Every element of $SU(N)$ is conjugate to a diagonal matrix, and so we can write
$$X 
= \begin{pmatrix}
\la_1& &0 \\
& \ddots& \\
0&& \la_N  \end{pmatrix},$$    
where $\la_i\in U(1)$ and $\la_1\cdots \la_N =1.$ However, because $X$ is conjugate to $\om X$, we must have 
$$\{\la_1,\ldots, \la_N\} = \{ \om \la_1, \ldots, \om \la_N \}.$$
Reordering the terms, we can arrange that 
$\la_i = \om^{i-1} \la_1$ for $i=1,\ldots, N$. Since $\det X =1$, we have
$\la_1^N = (-1)^{N-1},$ and so without loss of generality we can take 
$$\la_1 = \begin{cases} 1& \text{if $N$ is odd,} \\
\xi &\text{if $N$ is even.}
\end{cases}$$
This shows that $X$ is of the required form.

Next, observe that $XYX^{-1} Y^{-1} = \om I$ if and only if $XY = \om YX.$
Writing $Y = (y_{ij})$ and comparing the $(ij)$ entries on right and left, it follows that
$$\om^{i-1} y_{ij} = y_{ij} \om^{j}.$$
This implies that $y_{ij}=0$ unless $i \equiv j+1 \mod N$.
Furthermore, since $Y$ has only one nonzero entry in each row and column, each entry must lie in $U(1)$ and we find that
 $$Y = \begin{pmatrix}
0& \dots& & \mu_1 \\
\mu_2& \ddots& &\vdots\\
 &\ddots &  \\
     0&  &\mu_N &0 \end{pmatrix},$$
where $\mu_i \in U(1)$ satisfy $\mu_1\cdots \mu_N = (-1)^{N-1}$
(since $\det Y =1$).
Because $X$ is diagonal with $N$ distinct eigenvalues, the stabilizer subgroup $\Stab(X)$ is a copy of the standard maximal torus, i.e.
$$ \Stab(X) = \{ \diag(\th_1, \ldots, \th_N) \mid \th_i \in U(1), \th_1 \cdots \th_N =1 \}  
\cong T^{N-1}.$$ A matrix $P = \diag(\th_1, \ldots, \th_N) \in \Stab(X)$ acts on $Y$ by
$$P Y P^{-1} =   \begin{pmatrix}
0& \dots& & \th_1 \th_N^{-1} \mu_1  \\
\th_2 \th_1^{-1} \mu_2& \ddots& &\vdots\\
 &\ddots &  \\
     0&  & \th_N \th_{N-1}^{-1} \mu_N  &0 \end{pmatrix}.$$

Setting
\begin{align*}
\th_1 &=\mu_1^{-1}\\
\th_2 &= \th_1 \; \mu_2^{-1} = \mu_1^{-1} \mu_2^{-1} \\
 \th_3 &= \th_2 \; \mu_3^{-1}=  \mu_1^{-1} \mu_2^{-1}\mu_3^{-1}\\
& \;\; \vdots \\
\th_N &= \th_{N-1} \; \mu_N^{-1}= \mu_1^{-1}  \cdots \mu_N^{-1} =(-1)^{N-1},
\end{align*} 
 it follows that 
$$P Y P^{-1}=  \begin{pmatrix}
0& \dots& & (-1)^{N-1} \\
1& \ddots& &\vdots\\
 &\ddots &  \\
     0&  &1 &0 \end{pmatrix}.$$
   Since $PX P^{-1} =X, $ this shows
 that, up to conjugation, $Y$ is of the required form.
 Irreducibility of the pair $(X,Y)$ follows from the fact that $\Stab(X) \cap \Stab(Y) = \ZZ_N.$
\end{proof}

\begin{remark}
If $XYX^{-1} = \om Y$, then  $X^d Y X^{-d} = \om^d Y$ by induction.
This shows that if $(X,Y)$ are as in Theorem \ref{thm-pair}, then $(X^d,Y)$ satisfies
$[X^d,Y] = \om^d I$. Using this observation, one can show  that solutions $(X',Y')$ to  $[X',Y']= \om^d I$ are
irreducible and unique up to conjugation provided $d$ is relatively prime to $N$.
This fails if $d$ is not relatively prime to $N$; when $N=4$ and $d=2,$ one can construct  
non-conjugate families of pairs $(X,Y) \in SU(4)\times SU(4)$ satisfying
$[X,Y] = -I$. 
All of these pairs are reducible.
\end{remark}

We now use this to evaluate $h_{N,a}(H)$ for the Hopf link $H$.

\begin{proposition} \label{prop-hopf}
Suppose $H$ is the Hopf link and $1 \leq d < N$ is relatively prime to $N$. Then
$h_{N,a}(H) = \pm 1$ for $a = (d,N-d)$.
\end{proposition}
\begin{proof}
We motivate the proof with the following argument. The Hopf link $H$  has link group $G_H = \langle x,y \mid [x,y] = 1\rangle$, and Theorem \ref{thm-pair} implies there is a unique irreducible projective representation $\varrho \colon G_H \to SU(N)$ with
$[\varrho(x),\varrho(y)] = \om^d I.$ Uniqueness of $\varrho$ up to conjugacy implies that $h_{N,a}(H) = \pm 1$.

More precisely, notice that the Hopf link is the closure of the braid $ \si_1^2 \in B_2$ 
and fix the labels $a=(d, N-d)$ for $H$, where $1 \leq d < N$ is relatively prime to $N$. The braid 
$\si=\si_1^2$ acts on pairs $(X_1,X_2) \in \R_2 = SU(N) \times SU(N)$ in the usual way (see Example \ref{first-ex-page}), and for $\ep = (\ep_1,\ep_2)$ we have
$$\ep \si (X_1,X_2) =  (\ep_1 X_2^{-1} X_1 X_2, \ep_2 X_2^{-1} X_1^{-1} X_2 X_1 X_2).$$
For $(\ep_1,\ep_2)=(\om^d, \om^{N-d})$, one can easily see 
 that $(X_1,X_2) \in \Fix(\ep \si)$ if and only if 
$[X_1,X_2] = \om^d.$
By Theorem \ref{thm-pair} and the preceding remarks, this equation has one solution which is irreducible and unique up to conjugation. Lemma \ref{lemma-nondeg} below shows that the solution is non-degenerate, and this implies that $h_{N,a}(H) = \pm 1$ for the Hopf link.
\end{proof}

The next result establishes the non-degeneracy result required for the above computation of $h_{N,a}(H)$.

\begin{lemma} \label{lemma-nondeg}
Let $H$ be the Hopf link, $G_H$ its link group, and $1 \leq d < N$ relatively prime to $N$. Suppose $\varrho\colon G_H \to SU(N)$ is the projective representation, unique up to conjugation, of the link group $G_H$ with $a=(d,N-d)$. Then  $H^1(G_H;su(N)_{\Ad \varrho})=0.$
\end{lemma}

\begin{proof}
Let $Z=S^3 \sm \tau H$ be the link exterior, and 
recall that the exterior of every non-split link in $S^3$ is a $K(\pi,1).$  
Thus $H^i(Z; su(N)_{\Ad \varrho})=H^i(G_H; su(N)_{\Ad \varrho})$,
where $G_H= \pi_1(Z)$ is the link group.

For the Hopf link, the link group $G_H = \ZZ \times \ZZ$ is the free abelian group of rank two. Since $\varrho$ is irreducible, it follows that
$H^0(G_H; su(N)_{\Ad \varrho})=0$, and Poincar\'e duality implies that
$H^2(G_H; su(N)_{\Ad \varrho})=0$. Using $\chi(Z)=0$, this shows that $H^1(G_H; su(N)_{\Ad \varrho})=0,$ which completes the proof of the lemma.
\end{proof}

Next, we consider a chain link $L$ and we establish the following vanishing result for $h_{N,a}(L)$.

\begin{proposition} \label{prop-chain}
Suppose $L$ is a chain link with $n>2$ components and that $n=N$. Then $h_{N,a}(L)=0$ for $a=(d,\ldots, d)$, where $d$ is relatively prime to $N$. 
\end{proposition}

\begin{proof}
We start with the chain link $L$ with $n=3$ components. It has link group with presentation
$$G_L = \langle x,y,z \mid [x,y] = 1 = [y,z] \rangle.$$ We will parameterize all triples $(X,Y,Z) \in SU(3)\times SU(3)\times SU(3)$ satisfying $[X,Y] = \om I = [Y,Z]$, and we will use this to show that 
$h_{3,a}(L) = 0$ for $a=(1,1,1).$  

Applying Theorem \ref{thm-pair}, up to conjugacy, there is a unique irreducible pair $(X,Y) \in SU(3)\times SU(3)$ satisfying the equation $[X,Y]=\om I.$ This pair is given by
$$   X = \begin{pmatrix}
1&0& 0 \\
0& \om &0 \\
 0&  0 &\om^2 
 \end{pmatrix}
 \quad \text{ and }
Y = \begin{pmatrix}
0& 0 &1 \\
1& 0 & 0 \\
0& 1 &0 \end{pmatrix}.$$
In a general group, the commutator satisfies the relations
$$[x,y]^{-1} = [y,x] = y \, [x, y^{-1}]\, y^{-1} = x\, [x^{-1}, y] \, x^{-1}.$$
Setting $Z= X^{-1}$, this shows that $[Y,Z] = \om I$, and thus the triple $(X,Y,Z)$ gives rise to a projective representation $\varrho\colon G_L \to SU(3)$ satisfying $$[\varrho(x),\varrho(y)] = \om I = [\varrho(y),\varrho(z)].$$
If $P \in \Stab(Y)$, then 
$[Y,PZP^{-1}] = P [Y,Z] P^{-1} = \om I,$ hence the action of $\Stab(Y)$ on triples given by
$(X,Y,Z) \mapsto (X,Y,PZP^{-1})$ 
preserves the relations and is nontrivial on conjugacy classes. It follows that the solution set is 2-dimensional and parameterized by $\Stab(Y)/\ZZ_3,$ which has Euler characteristic zero since $\Stab(Y) \cong T^2$ is a copy of a maximal torus. A calculation similar to the one in the proof of Lemma \ref{lemma-nondeg} shows that the solution set is a nondegenerate critical submanifold, and a standard argument then shows that its contribution to the invariant is given by $\pm$ its Euler characteristic (cf. the proof of Proposition 8 in \cite{BH}). It follows that 
$h_{3,a}(L) =0$ for $a=(1,1,1),$ and a similar argument shows that $h_{3,a}(L) = 0$ for $a=(2,2,2).$
        
One can also prove this via a direct approach making use of the fact that  $L$ is the closure of the braid $\si = \si_1^2 \si_2^2$ and parameterizing the fixed point set $\Fix(\ep \si)$ as was done for the Hopf link. We leave the details to the reader.

Next, consider the chain link $L$ with 4 components. It has link group with presentation
$$G_L = \langle x,y,z, w \mid [x,y] = 1 = [y,z] = [z,w] \rangle.$$  
By Theorem \ref{thm-pair}, up to conjugacy, there is a unique irreducible pair $(X,Y) \in SU(4)\times SU(4)$ satisfying the equation $[X,Y]=\om I.$ This pair is given by
$$   X = \begin{pmatrix}
\xi&0&0& 0 \\
0& \xi^3&0&0   \\
 0& 0 &\xi^5&0 \\
 0&0&0& \xi^7
 \end{pmatrix}
 \quad \text{ and }
Y = \begin{pmatrix}
0& 0 & 0& -1 \\
1& 0 & 0 &0\\
0& 1 &0 &0\\
0&0&1&0 \end{pmatrix}.$$
Taking $Z=X^{-1}$ and $W=Y^{-1},$ one can show that the 4-tuple
$(X,Y,Z,W)$ gives rise to a projective representation $\varrho\colon G_L \to SU(3)$ with $[\varrho(x),\varrho(y)] = \om I = [\varrho(y),\varrho(z)] = [\varrho(z),\varrho(w)]$. The two groups $\Stab(Y)$ and $\Stab(Z)$ act on 4-tuples
by $(X,Y,Z,W) \mapsto (X,Y,PZP^{-1}, PWP^{-1})$ for $P \in \Stab(Y)$ and 
$(X,Y,Z,W) \mapsto (X,Y,Z, QWQ^{-1})$ for $Q \in \Stab(Z)$, and these actions preserve the relations and are nontrivial on conjugacy classes. It follows that the solution set is 6-dimensional and parameterized by $\Stab(Y)/\ZZ_4 \times \Stab(Z)/\ZZ_4,$ which has Euler characteristic zero.
By similar considerations as in the previous case, it follows that 
$h_{4,a}(L) =0$ for $a=(1,1,1,1),$ and
a similar argument shows that $h_{4,a}(L) =0$ for $a=(3,3,3,3).$ 

As before, one can perform these computations directly by noting that $L$ is the closure of the braid $\si = \si_1^2 \si_2^2 \si_3^2$ and parameterizing the fixed point set $\Fix(\ep \si)$.

 This argument generalizes to the $n$-component chain link in a straightforward manner, as we now explain.
The chain link $L$ with $n$ components has link group with presentation
$$G_L = \langle x_1,\ldots ,x_n \mid [x_1,x_2] = \cdots = [x_{n-1},x_n] = 1 \rangle.$$  
By Theorem \ref{thm-pair}, up to conjugacy, there is a unique irreducible pair $(X_1,X_2) \in SU(N)\times SU(N)$ satisfying the equation $[X_1,X_2]=\om I.$ A solution is obtained by taking $X_1=X$
and $X_2=Y$ for $X,Y$ as in the statement of the theorem, and setting $X_{i+2}=X_{i}^{-1}$ for
$i=1,\ldots, n-2$, the $n$-tuple $$(X_1,\ldots, X_n) \in SU(N) \times \cdots \times SU(N)$$ 
 is easily seen to satisfy the relations
$$ [X_1,X_2] = \cdots = [X_{n-1},X_n] = \om I.$$
For $i= 3, \ldots,  n$, the group $\Stab(X_i)$ acts on $n$-tuples by 
$$(X_1,\ldots, X_n) \mapsto (X_1,\ldots, X_i, PX_{i+1}P^{-1}, \ldots, PX_nP^{-1})$$ for $P \in \Stab(X_i).$
These actions preserve the relations and are nontrivial on conjugacy classes. Since each $\Stab(X_i) \cong T^{N-1}$ is a maximal torus, it follows that the solution set has dimension $(N-1)(N-2)$ and is  parameterized by $\Stab(X_3)/\ZZ_N \times \cdots \times \Stab(X_n)/\ZZ_N,$ which has Euler characteristic zero.
It follows that 
$h_{N,a}(L) =0$ for $a=(1,\ldots, 1),$ and
a similar argument shows that $h_{N,a}(L) =0$ for $a=(d,\ldots, d)$ for any $d$ relatively prime to $N$.
 \end{proof}
  
\subsection{Split Links}
In this section, we consider links $L  \subset S^3$ that are geometrically split and prove a vanishing result
for $h_{N, a}(L)$ provided that the labels satisfy the following condition.
Assume $L$ is a link with $n$ components, and suppose it is split. Denoting the components of $L$ by $ \ell_1 \cup \cdots \cup \ell_n$, this means that  $L = L_1 \cup L_2$, where up to reordering $L_1 =   \ell_1 \cup \cdots \cup \ell_{n_1}$ and $L_2 =  \ell_{n_1+1} \cup \cdots \cup \ell_n$ are sublinks that are separated by a 2-sphere.
We shall assume that the labels $(a_1,\ldots, a_n)$ satisfy the additional condition:
\begin{equation}\label{NoMult}
a_1+ a_2+\cdots + a_{n_1} \text{ is not a multiple of $N$.}
\end{equation} 

Using Markov moves we can always find a {\it split braid} representative $\be \in B_k$ of $L$, see Figure \ref{SplitBraid}. This means that
$\be= \be_1 \be_2$ where $\be_1 \in {\rm Im}(B_{k_1} \stackrel{i_1}{\hookrightarrow} B_k)$ and 
$\be_2 \in {\rm Im}(B_{k_2} \stackrel{i_2}{\hookrightarrow} B_k)$ for $k=k_1+k_2$ and 
$i_1, i_2$ are injective maps obtained by stabilizing on the right and left, respectively. More precisely, $i_1$ takes a braid in $B_{k_1}$ and adds $k_2$  trivial strands on the right, and $i_2$ takes a braid in $B_{k_2}$ and adds $k_1$ trivial strands on the left.   
Any link $L$ obtained as the closure $\wh\be$ of a split braid is obviously a split link, and any split link $L$ can be obtained as the closure of a split braid.

\begin{figure}[h]
\includegraphics[scale=1]{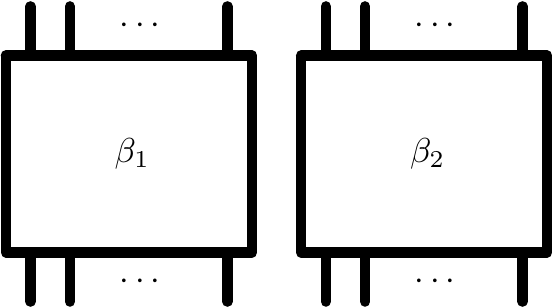} 
\caption{A split braid} \label{SplitBraid}
\end{figure}

\begin{proposition}\label{SplitLinks}
Suppose $L$ is a split link and that $\be$ is a split braid with $\wh\be = L.$ Suppose further that $a=(a_1,\ldots, a_n)$ is an $n$-tuple of  labels satisfying Equation \eqref{NoMult}, and $\ep=(\ep_1,\ldots, \ep_k)$ is a compatible $k$-tuple.
Then the intersection $ \De_k \cap  \Ga_{\ep\be}=\varnothing$, and consequently  $h_{N,a}(L) = 0.$
\end{proposition}

\begin{proof}

Let $X \in  \De_k \cap  \Ga_{\ep\be}$, then by \eqref{Ep}
$$X_1 \cdots X_{k_1} =  \om^{dn_1}\be(X)_1 \cdots \be(X)_{k_1}.$$
Since $\be= \be_1 \be_2$ is a split braid with $\be_1 \in B_{k_1}$, by \eqref{Product} we have that 
$$ \be(X)_1 \cdots \be(X)_{k_1}= \be_1(X)_1 \cdots \be_1(X)_{k_1}= X_1 \cdots X_{k_1},$$
and this implies
$$X_1 \cdots X_{k_1} = \om^{a_1+\cdots +a_{n_1}}X_1 \cdots X_{k_1}.$$
But $\om^{a_1+\cdots +a_{n_1}} \neq 1$ by the assumption of Equation \eqref{NoMult}, and this gives the desired contradiction.
\end{proof}

\subsection{Concluding remarks}
One can give an alternative interpretation of the invariants $h_{N,a}(L)$ in terms of a signed count of conjugacy classes of representations $\varrho \colon G_L \to PU(N)$ of the link group as follows. We begin by recalling the classification of principal $PU(N)$ bundles from \cite{Woodward}. 
The classifying space $BPU(N)$ is simply connected and has  $\pi_2(BPU(N)) = \ZZ_N,$ and an application of the main theorem of \cite{Woodward} implies that principal $PU(N)$ bundles $P \to X$ over a 3-complex $X$ are determined by the characteristic class $w(P) \in H^2(X;\ZZ_N)$. In case $N=2$, $PU(2) = SO(3)$ and $w(P)$ coincides with the second Stiefel-Whitney class. 

Let $L \subset S^3$ be a link and $M_L = S^3 \sm \tau(L)$ its exterior. A projective $SU(N)$ representation induces a representation $\varrho\colon G_L \to PU(N)$, and we denote the associated cohomology class by $w(\varrho) \in H^2(G_L;\ZZ_N)$. The class $w(\varrho)$ vanishes if and only if $\varrho$ lifts to an $SU(N)$ representation. Further, there is a canonical injection $H^2(G_L;\ZZ_N) \to H^2(M_L;\ZZ_N)$ which is an isomorphism if and only if $M_L$ is aspherical, i.e.~ if and only if $L$ is non-split.

By reduction mod $N$, an allowable $n$-tuple $(a_1,\ldots, a_n) \in \ZZ^n$ determines a unique cohomology class $\bar{w}(a_1,\ldots, a_n) \in H^2(\del M_L;\ZZ_N) \cong (\ZZ_N)^n.$ 
The exact sequence in cohomology for the pair $(M_L,\del M_L)$ gives 
$$\to H^2(M_L;\ZZ_N) \stackrel{i^*}{\lto} H^2(\del M_L;\ZZ_N)\to H^3(M_L,\del M_L;\ZZ_N) \to 0;$$
and condition (iii) from \S \ref{subsect2-2}  guarantees that $\bar{w}(a_1,\ldots, a_n)$ lies in the image of $i^*$ and hence determines a class $w(a_1,\ldots, a_n) \in H^2(M_L;\ZZ_N)$. Condition (ii) implies that the class $w(a_1,\ldots, a_n)$ has order $N$.  

From this point of view, the invariant $h_{N,a}(L)$ is closely related to the signed count of conjugacy classes of representations $\varrho \colon G_L \to PU(N)$ such that $w(\varrho) = w(a_1,\ldots, a_n).$ Proposition \ref{SplitLinks} is therefore a direct consequence of the fact that for split links $L$ and allowable $n$-tuples $(a_1,\ldots, a_n)$ satisfying condition \eqref{NoMult}, the associated cohomology class $w(a_1,\ldots, a_n)$, which is not unique, does not lie in the image of $H^2(G_L;\ZZ_N) \to H^2(M_L;\ZZ_N)$.

As mentioned in the introduction, it would be interesting to investigate the relationship between the $SU(N)$ Casson-Lin invariants studied here and the $SU(N)$ instanton Floer groups constructed by Kronheimer and Mrowka in \cite{KM1, KM2}. It would also be interesting to understand the relationship between the $SU(N)$ Casson-Lin invariants and classical link invariants. For example, the main result of \cite{HS1} equates the $SU(2)$ Casson-Lin invariant $h_{2}(L)$ of a two component link $L=\ell_1 \cup \ell_2$ with the linking number $\lk(\ell_1,\ell_2).$ The following conjecture, if true, would give a generalization to the higher rank setting.

\begin{conjecture} If $L=\ell_1 \cup \ell_2$ is a two component link in $S^3$, then the $SU(N)$ Casson-Lin invariant satisfies $$h_{N,a}(L) =\lk(\ell_1,\ell_2)^{N-1}.$$
\end{conjecture}

This conjecture is consistent with all known computations of the $SU(N)$ Casson-Lin invariants, and it's possible that the invariants $h_{N,a}(L)$ are generally invariant under link homotopy. We hope to explore these topics in future work.

\bigskip
\noindent
{\it Acknowledgements.} Both authors would like to thank Chris Herald, Andy Nicas, and Nikolai Saveliev for their valuable input. We would also like to thank an anonymous referee for many helpful suggestions.   The first author expresses his gratitude to the Max Planck Institute for Mathematics in Bonn for its support.


\begin{thebibliography}{999}
\bibitem{A}
J.~Alexander, 
\emph{A lemma on systems of knotted curves}, Proc. Nath. Acad. Sci. USA
\textbf{9} (1923), 93--95.

\bibitem{Birman}   
J.~Birman, 
Braids, links, and mapping class groups. Princeton University Press, 
Princeton, 1974.

\bibitem{BH} H.~ U.~ Boden and C.~ M.~ Herald,
    \emph{A connected sum formula for the SU(3) Casson invariant,}
    J. Diff. Geom. \textbf{53} (1999), 443--465.



\bibitem{Collin-Steer}
O.~Collin, B.~Steer,
\emph{Instanton Floer homology for knots via $3$-orbifolds}, 
J. Differential Geom. \textbf{51} (1999), 149--202.

\bibitem{FRR}
R.~Fenn, R.~Rimanyi, and C.~Rourke
\emph{The braid-permutation group}, 
Topology \textbf{36} (1997), 123--135.

\bibitem{Floer}
A.~Floer,
\emph{An instanton-invariant for $3$-manifolds}, Comm. Math. Phys. 
\textbf{118} (1988), 215--240.

\bibitem{HS1}
E.~Harper, N.~Saveliev,
\emph{A Casson-Lin type invariant for links},
Pacific J. Math. \textbf{248} (2010), 139--154.

\bibitem{HS2}
E.~Harper, N.~Saveliev,
\emph{An instanton Floer homology for links}, 
 J. Knot Theory and Its Ramific. \textbf{21}(5) (2012), 1250054 (8 pages)


\bibitem{Herald}
C.~Herald,
\emph{Flat connections, the Alexander invariant, and Casson's invariant},
Comm. Anal. Geom. \textbf{5} (1997), 93--120.

\bibitem{HK}
M.~Heusener, J.~Kroll,
\emph{Deforming abelian $SU(2)$-representations of knot groups}, 
Comment. Math. Helv. \textbf{73} (1998), 480--498.

\bibitem{KT}   
C.~Kassel and V.~Turaev, 
Braid Groups. Springer Graduate Texts in Mathematics {\bf 247}, 2008.


\bibitem{KM1}
P.~Kronheimer, T.~Mrowka,
\emph{Knot homology groups from instantons}, J. Topology {\bf 4} (2011), 835--918. 

\bibitem{KM2}
P.~Kronheimer, T.~Mrowka,
\emph{Khovanov homology is an unknot detector}, Publ. IH\'{E}S {\bf 113} (2011), 97--208. 

\bibitem{Lin}       
X.-S.~Lin, \emph{A knot invariant via representations spaces},
J. Differential Geom. \textbf{35} (1992), 337--357.

\bibitem{Long}
D.~Long, \emph{On the linear representation of braid groups}, 
Trans. Amer. Math. Soc. 311 (1989) 535-560.

\bibitem{RS}        
D.~Ruberman, N.~Saveliev, \emph{Rohlin's invariant and gauge theory, I. Homology 3-tori}, 
Comm. Math. Helv. \textbf{79} (2004), 618--646.

\bibitem{Taubes}
C.~Taubes,
\emph{Casson's invariant and gauge theory}, J. Diff. Geom. 
\textbf{31} (1990), 547--599.

\bibitem{Woodward}
L.~M.~Woodward,
\emph{The classification of principal PU$_n$-bundles over a 4-complex}, J. Lond. Math Soc. (2)
\textbf{25} (1982), 513--524.

\end{thebibliography}
\end{document}